\documentclass[english]{amsart}
\usepackage{babel}

\usepackage{setspace}

\usepackage{mathtools}

\newtheorem{thm}{Theorem}[]

\newtheorem{lem}{Lemma}

\newcommand{\half}{\mbox{$\frac{1}{2}$}}

\newcommand{\IR}{\text{${\mathbb{R}}$}}

\newcommand{\IZ}{\text{${\mathbb{Z}}$}}

\newcommand{\comment}[1]{}

\title{Decay Properties of Spatial Molecular Orbitals}

\author{R. A. Zalik}

\address{Department of Mathematics and Statistics, Auburn University,
Auburn, Al. 36849--5310}\email{zalikri@auburn.edu}

 \thanks{arXiv 2404.14420}
\date{}

\begin{document}

\begin{abstract} 
{Using properties of the Fourier transform we prove that if a Hartree-Fock molecular spatial orbital is in $L_1(\IR^3)$, then it decays to zero as its argument diverges to infinity. The proof is rigorous, elementary and short. Our result implies that occupied orbitals with positive eigenvalues will decay to zero provided they are in $L_1$. 
}
\end{abstract}
\maketitle

\pagestyle{myheadings}
\markboth{Orbitals}{ZALIK}

\section{Introduction} \label{section1}
We begin with a statement of the Hartree--Fock (or HF) equations, their importance, and the relevance of our result in Quantum Chemistry. These topics are studied in many introductory texts, such as \cite{IL,SO}.

The HF equations are nonlinear eigenvalue equations of the form
\begin{multline} \label{eq1}
\left [-\half\nabla^2 - \sum_{j=1}^n \frac{Z_j}{|r-\xi_j|} \right] \psi_a(r)+2\sum_{c=1}^n \int |\psi_c(s)|^2 \frac{1}{|r-s|} \,ds\,\psi_a(r)\\
-\sum_{c=1}^n \int \psi_c^{\ast}(s)\psi_a(s) \frac{1}{|r-s|}\, ds\, \psi_c(r)=\varepsilon_a \psi_a(r),  \qquad a=1, \dots, n,
\end{multline}
where $Z_j$ is the atomic number of nucleus $N_j$, which is located at point $\xi_j \in \IR^3$ and $\nabla^2$  is the Laplace operator.
These equations apply only to systems with an even number $N$ of electrons, and $n=N/2$.
Their solutions are called HF approximations; they are approximate solutions of the electronic Schr\"odinger equation and provide a gateway to more accurate approximations. The solution of  \eqref{eq1} yields a sequence of  functions $\psi_a$ --called spin 
orbitals -- with orbital energies $\varepsilon_a$;  they are continuous on $\IR^3$ and $||\psi_a||_{\infty} \le 1$, and are usually assumed to form an orthonormal set in $L_2$. The $N$ spin orbitals with lowest energy are called \emph{occupied} spin orbitals

The HF equations may also be used on ions. If an  HF calculation is performed on a multiply charged anion, even some of the occupied orbitals $\psi_a$ will have positive eigenvalues. This amounts to a prediction, using a theorem of Koopmans (1934), that the multiply charged anion will spontaneously lose electrons. Our result implies that the occupied orbitals with positive eigenvalues will still decay to zero provided they are in $L_1$. 

We will use standard mathematical notation. In particular, $\IZ$ will denote the set of integers and $\IR$ the set of real numbers;
if $r =(x_1,x_2,x_3)\in \IR^3$, then $|r|=(x_1^2+x_2^2+x_3^2)^{\half}$; if $z$ is a complex number, $z^{\ast}$ will denote its complex conjugate, and $\int f$ will stand for $\int_{\IR^3} f$. For $1 \le p <\infty$ we say that $f$ is in $L_p(\IR^3)$ (which we will abbreviate as $L_p$), if
\[
\int_{\IR^3} |f(r)|^p\,dr < \infty,
\]
and the $L_p(\IR^3)$ norm of $f$ is defined by
\[
||f\||_p =\left (\int_{\IR^3} |f(r)|^p \, dr \right )^{1/p}.
\]
We say that $f$ is in $L_{\infty}(\IR^3)$ (which we will abbreviate as $L_{\infty}$) if there is a constant $C$ such that
$|f(r)| \le C$ for all $r \in \IR^3$. If $f$ is in $L_{\infty}$, the $L_{\infty}$ norm is defined as
\[
||f||_{\infty} = \sup|f(r)|, \qquad r \in \IR^3,
\]
(where``$\sup$"is the least upper bound). We also define $h(r)=|r|^{-1}$; $\mathfrak{F}[f]$ or $\widehat{f}$ will stand for the Fourier transform of $f$; thus, if for instance $f \in L_1$, then
\[
\widehat{f}(r) = \int_{\IR^3} f(t) e^{-2\pi i r \cdot  t}\,dt.
\]
The convolution $f\ast g$ of $f$ and $g$ is defined by
\[
(f\ast g)(r)=\int f(s) g(r-s)\, ds =\int f(r-s) g(s)\, ds.
\]

The following auxiliary proposition, of some independent interest, shows that \eqref{eq1} is well posed:
\begin{lem} \label{lem1}
If $f \in L_1 \cap L_{\infty}$, then
\[
||f\ast h||_{\infty} \le (4/\sqrt{3}) \pi ||f||_{\infty}+||f||_1;
\]
therefore, if $f,g \in L_2 \cap L_{\infty}$
\[
||(fg)\ast h||_{\infty} \le
(4/\sqrt{3}) \pi ||f||_{\infty}  ||g||_{\infty} +||f||_2||\,||g||_2.
\]
In particular, 
\[
\left |\int \psi_a^{\ast}(s)\psi_b(s) \frac{1}{|s-r|}\, ds \right |  \le (4/\sqrt{3}) \pi +1, \qquad a, b=1, \dots, n.
\]
\end{lem}
\begin{proof}
We have
\[
|(f \ast h)(r)| \le \int_{|s|\le1} |f(r-s)|\, |s|^{-1}\,ds+ \int_{|s| >1} |f(r-s)|\, |s|^{-1}\,ds  = I_1+I_2.
\] 
Applying H\"older's inequality and switching to spherical coordinates we get
\[
I_1 \le \left ( \int_{|s| \le1} |f(r-s)|^2\, ds \int_{|s|\le1} |s|^{-2}\,ds\right )^{\half} \le (4/\sqrt{3})  \pi ||f||_{\infty}. 
\]
On the other hand,
\[
I_2 \le ||f||_1.
\]
\end{proof}
The proof of the following theorem relies on basic properties of the Fourier transform (\cite{Fo, SS}).
\begin{thm}\label{thm1}
Let $\psi_a \in L_1$; then $\lim_{r \rightarrow \infty}\psi_a(r) =0$.
\end{thm}
\begin{proof}
Equation \eqref{eq1} may be written in the equivalent form
\begin{equation} \label{eq2}
\mathfrak[\nabla^2\psi_a](r)=g_a(r),
\end{equation}
where 
\[
g_a(r)=  - 2\sum_{j=1}^n \frac{Z_j}{|r-\xi_j|}\psi_a(r) +4 A \psi_a(r) - 2\sum_{c=1}^n B_{a,c} \psi_c(r) -2\varepsilon_a \psi_a(r)
\]
with
\[
A= \sum_{c=1}^n \int |\psi_c(s)|^2 \frac{1}{|r-s|} \,ds
\]
and
\[
B_{a,c}= \int \psi_c^{\ast}(s)\psi_a(s) \frac{1}{|r-s|}\, ds.
\]

On the other hand, let $K$ be arbitrary but fixed and such that $|\xi_j| \le K, j=1, \dots, n$; then
\begin{multline*}
\int |r-\xi_j|^{-2} |\psi_a(r)|^2 \,dr = \\
\int_{|r| \le K+1} |r-\xi_j|^{-2} |\psi_a(r)|^2 \,dr + \int _{|r| > K+1} |r-\xi_j|^{-2} |\psi_a(r)|^2 \,dr = J_1 + J_2.
\end{multline*}
We have:
\[
J_1 \le \int_{|r| \le K+1} |r-\xi_j|^{-2}\,dr \le \int_{|r| \le K+1 + |\xi_j|} |r|^{-2}\,dr = 4 \pi (K+1 + |\xi_j| )
\]
and, since $|r|>K+1$ implies that $ |r-\xi_j| \ge 1$,
\[
J_2  \le \int _{|r| > K+1} |\psi_a(r)|^2 \,dr \le ||\psi_a||^2_2,
\]
which implies that $|r-\xi_j|^{-1}\psi_a(r)$ is in $L_2$ for all $1 \le j \le n$, which in turn implies that $g(r) \in L_2$.
But  $\psi_a \in L_1$ by hypothesis; thus
\[
\mathfrak[\nabla^2\psi_a](r)=-4\pi^2|r|^2 \widehat{\psi}_a(r).
\]
Since  $\psi_a$ is bounded, it is in $L_2$ as well (and therefore so is $\widehat{\psi}_a$); thus \eqref{eq2}  implies that $(1+|r|^2)\widehat{\psi}_a \in L_2$. Since
\[
\widehat{\psi}_a(r)=\frac{1}{1+|r|^2}(1+|r|^2)\widehat{\psi}_a(r)
\]
and $\frac{1}{1+|r|^2} \in L_2$, we deduce that $\widehat{\psi}_a \in L_1$, whence 
\[
\lim_{r\rightarrow \infty} \mathfrak{F}[\widehat{\psi}_a](r)=0. 
\]
Since $\psi_a$ is continuous on $\IR^3$ we know that 
$\psi_a(r)=\mathfrak{F}[\widehat{\psi}_a](-r)$, and the assertion follows.
\end{proof}
\section{Acknowledgement}
\thanks{The author is grateful to Professor Vincent Ortiz, Department of Chemistry and Biochemistry, Auburn University. The explanation of the relevance of Theorem 1 in Quantum Chemistry is based on his helpful comments.}

\end{document}